\documentclass[12pt]{amsart}

\usepackage{algorithmic}
\usepackage[section]{algorithm}

\usepackage{fullpage}

\usepackage{amsfonts,amssymb,amscd,amsmath,enumerate,verbatim,color,xcolor}
\usepackage[latin1]{inputenc}
\usepackage{amscd}
\usepackage{latexsym}
\usepackage{tikz}
\usepackage{graphicx} 
\usepackage{here}
\usepackage{mathptmx}

\usepackage{hyperref}
\hypersetup{colorlinks,linkcolor=blue ,citecolor=blue, urlcolor=blue}
\def\NZQ{\mathbb}               

\def\ZZ{{\NZQ Z}}
\def\RR{{\NZQ R}}

\def\KK{{\NZQ K}}

\def\frk{\mathfrak}               

\def\Phi{{\frk N}}
\def\ab{{\mathbf a}}
\def\bb{{\mathbf b}}

\def\eb{{\mathbf e}}
\def\tb{{\mathbf t}}

\def\wb{{\mathbf w}}
\def\xb{{\mathbf x}}

\def\opn#1#2{\def#1{\operatorname{#2}}} 
\opn\gr{gr}

\def\Mc{{\mathcal M}}

\def\Sc{{\mathcal S}}

\def\Pc{{\mathcal P}}

\def\Mc{{\mathcal M}}

\newtheorem{Theorem}{Theorem}[section]
\newtheorem{Lemma}[Theorem]{Lemma}
\newtheorem{Corollary}[Theorem]{Corollary}
\newtheorem{Proposition}[Theorem]{Proposition}

\theoremstyle{definition}

\newtheorem{Example}[Theorem]{Example}

\newtheorem{Conjecture}[Theorem]{Conjecture}

\let\epsilon\varepsilon
\let\phi=\varphi
\let\kappa=\varkappa
\opn\dis{dis}
\opn\height{height}
\opn\dist{dist}
\def\pnt{{\raise0.5mm\hbox{\large\bf.}}}

\opn\Lex{Lex}
\opn\conv{conv}

\numberwithin{equation}{section}

\title{Toric ideals of matching polytopes and edge colorings}
\author{Kenta Mori, Ryo Motomura, Hidefumi Ohsugi and Akiyoshi Tsuchiya}

\address{Kenta Mori,
	Department of Mathematical Sciences,
	School of Science,
	Kwansei Gakuin University,
	Sanda, Hyogo 669-1330, Japan}
\email{k-mori@kwansei.ac.jp}

\address{Ryo Motomura,
Department of Information Science,
Faculty of Science,
Toho University,
2-2-1 Miyama, Funabashi, Chiba 274-8510, Japan} 
\email{6524012m@st.toho-u.ac.jp}

\address{Hidefumi Ohsugi,
	Department of Mathematical Sciences,
	School of Science,
	Kwansei Gakuin University,
	Sanda, Hyogo 669-1330, Japan} 
\email{ohsugi@kwansei.ac.jp}

\address{Akiyoshi Tsuchiya,
Department of Information Science,
Faculty of Science,
Toho University,
2-2-1 Miyama, Funabashi, Chiba 274-8510, Japan} 
\email{akiyoshi@is.sci.toho-u.ac.jp}

\begin{document}

\begin{abstract}
In the present paper, we investigate the maximal degree of minimal generators of the toric ideal of the matching polytope of a graph.
It is known that the toric ideal associated to a bipartite graph is generated by binomials of degree at most $3$.
We show that this fact is equivalent to a result in the theory of edge colorings of bipartite multigraphs.
Moreover, a characterization of bipartite graphs whose toric ideals are generated by quadratic binomials is given.
Finally, we discuss the maximal degree of minimal generators of the toric ideal associated to a general graph and give a conjecture.
\end{abstract}

\maketitle

\section{Introduction}
In the present paper, we discuss a relationship between an algebraic property of toric ideals arising from graphs and a combinatorial property of edge colorings of multigraphs.
Throughout this paper, we assume that a graph is simple, namely, it has no loops and no multiple edges and, a multigraph has no loops.
Let $G$ be a graph with the vertex set $V(G)=[d]:=\{1,2,\ldots,d\}$ and the edge set $E(G) = \{e_1,\dots, e_n\}$.
A \textit{matching} of $G$ is a set of pairwise non-adjacent edges of $G$, and
a \textit{perfect matching} of $G$ is a matching that covers every vertex of $G$.
Let $M(G)$ (resp.~$PM(G)$) denote the set of all matchings (resp.~perfect matchings) of $G$.
Given a subset $M \subset E(G)$, we associate the $(0,1)$-vector $\rho(M)=\sum_{e_j \in M} \eb_j \in \RR^n$. 
Here $\eb_j$ is the $j$th unit coordinate vector in $\RR^n$. For example, $\rho(\emptyset)=(0,\ldots,0) \in \RR^n$.
Then the (\textit{full}) \textit{matching polytope} $\mathcal{M}_G$ of $G$ is defined as the convex hull
\[
\mathcal{M}_G=\conv\left\{\rho(M) : M \in M(G)\right\}
\]
and the \textit{perfect matching polytope} $\mathcal{P}_G$ of $G$ is defined as
\[
\mathcal{P}_G=\conv\left\{\rho(M) : M \in PM(G)\right\}.
\]
Note that $\mathcal{P}_G$ is a face of $\Mc_G$.
Moreover, the perfect matching polytope of a complete bipartite graph $K_{d,d}$ is called the \textit{Birkhoff polytope}, denoted by $\mathcal{B}_d$.

In \cite{Diaconis2006}, it was conjectured that the toric ideal $I_{\mathcal{B}_n}$ of the Birkhoff polytope $\mathcal{B}_n$ is generated by binomials of degree at most $3$, and this conjecture was shown in \cite{Yamaguchi2014}.
Moreover, in \cite{Domokos2016}, by using this result, the toric ideal of a flow polytope is generated by binomials of degree at most $3$.
For a homogeneous ideal $I$, let $\omega(I)$ denote the maximal degree of  minimal generators of $I$.
Since the matching polytope of a bipartite graph is unimodularly equivalent to a flow polytope (see Appendix \ref{sect:app}), the following result holds:
\begin{Theorem}[\cite{Domokos2016}]
\label{thm:bip_upper}
    For a bipartite graph $G$, one has $\omega(I_{\Mc_G}) \leq 3$.
\end{Theorem}

Next, we recall a result of edge-colorings of multigraphs.
Let $G$ be a multigraph. 
For a $k$-edge-coloring $f$ of $G$ and a color $1 \leq j \leq k$, 
let $M^{(e)}(f,j)$ denote the set of all edges of color $j$. 
We say that two $k$-edge-colorings $f$ and $g$ of $G$ \textit{differ by an $m$-colored subgraph} if there is a set of colors $S$ of size $m$ such that $M^{(e)}(f,j) \neq M^{(e)} (g,j)$ for each $j \in S$, but $M^{(e)}(f,j)=M^{(e)}(g,j)$ for each $j \notin S$.
For two $k$-edge-colorings $f,g$ of $G$, we write $f \sim_r g$ if there exists a sequence $f_0,f_1,\ldots,f_s$ of $k$-edge-colorings of $G$ with $f_0=f$ and $f_s=g$ such that $f_i$ differs from $f_{i-1}$ by a $k_i$-colored subgraph with $k_i \leq r$. 
Note that $f \sim_r g$ implies $f \sim_{r+1} g$.
In \cite{Asratian1998,Asratian2009,Asratyan1991}, the following result was shown:
\begin{Theorem}[\cite{Asratian1998,Asratian2009,Asratyan1991}]
\label{thm:bip_differ}
Let $G$ be a bipartite multigraph.
Then for any $k$-edge-colorings $f$ and $g$ of $G$, one has $f \sim_3 g$.
\end{Theorem}

In the present paper, we show that Theorems \ref{thm:bip_upper} and \ref{thm:bip_differ} are equivalent.
For a simple graph $G$ on $[d]$ with $E(G)=\{e_1,e_2,\ldots,e_n\}$ and $\ab=(a_1,\ldots,a_n) \in \ZZ^n_{\geq 0}$, let $G^{(e)}_{\ab}$ be the multigraph on $[d]$ such that $G^{(e)}_{\ab}$ has $a_i$ multiedges $e_i$ for each $i$. We call $G^{(e)}_{\ab}$ the \textit{edge-replication multigraph} of $G$ on $\ab$.
Then our main result is the following:
\begin{Theorem}\label{thm:degree}
Let $G$ be a graph with $n$ edges.
Then $\omega(I_{\Mc_G}) \leq r$ if and only if for any $\ab \in \ZZ_{\geq 0}^n$ and for any $k$-edge-colorings $f$ and $g$ of $G^{(e)}_{\ab}$, one has $f \sim_r g$.\end{Theorem}

Since any edge-replication multigraph of a simple bipartite graph is bipartite, Theorems \ref{thm:bip_upper} and \ref{thm:bip_differ} are equivalent from this theorem. In particular, this gives an alternative proof of the conjecture of \cite{Diaconis2006} mentioned above.
Moreover, we can show a similar result for perfect matching polytopes (Proposition \ref{pm gene}).

On the other hand, we give a characterization of a bipartite graph such that $\omega(I_{\Mc_G})=2$, i.e., $I_{\Mc_G}$ is generated by quadratic binomials.
In fact,
\begin{Theorem} \label{thm:bip_quad}
    Let $G$ be a bipartite graph. 
    Then the following conditions are equivalent{\em :}
    \begin{itemize}
        \item[{\rm (i)}]$\omega(I_{\Mc_G})=2${\rm ;} 
        \item[{\rm (ii)}]
        $G$ has no odd subdivision of $K_{2,3}$ as a subgraph{\rm ;} 
        \item[{\rm (iii)}]
         each block of $G$ is a bipartite graph having no odd subdivision of $K_{2,3}$ as a subgraph.
    \end{itemize}
Otherwise, one has $\omega(I_{\Mc_G})=3$.
\end{Theorem}
Finally, we discuss the maximal degree of minimal generators of $I_{\mathcal{M}_G}$ for a general graph $G$.
There exists a non-bipartite graph $G$ with $\omega(I_{\mathcal{M}_G})=4$ (see Example \ref{exam:deg4}).
Hence we cannot extend Theorem \ref{thm:bip_differ} to the case of arbitrary graphs.
However, motivated by a conjecture in \cite{Asratian1998}, we give the following conjectures:
\begin{Conjecture}
\label{conj:non-perfect}
    Let $G$ be a (non-bipartite) graph. Then one has $\omega(I_{\Mc_G}) \leq 4$.
\end{Conjecture}

\begin{Conjecture}\label{conj:non-perfect-matching}
       Let $G$ be a (non-bipartite) graph. Then one has $\omega(I_{\Pc_G}) \leq 4$.
\end{Conjecture}
In general, we have $\omega(I_{\Pc_G}) \le \omega(I_{\Mc_G})$
(Proposition \ref{p to m} (a)).
On the other hand, there exists a graph $G$ such that $\omega(I_{\Mc_G}) \neq \omega(I_{\Pc_G})$. However, we can show that these conjectures are equivalent (Proposition \ref{prop:conjs}).

The present paper is organized as follows:
In Section \ref{sect:pre}, we introduce the toric ideal of a lattice polytope and define the stable set polytope associated with a graph. Note that $\Mc_G$ is the stable set polytope of the line graph of $G$. Section \ref{sect:stable} proves Theorem \ref{thm:degree} by providing an upper bound on the maximal degree of minimal generators of the toric ideal of a stable set polytope using vertex-colorings. 
Section \ref{sect:method} gives an algebraic method to determine if $f \sim_r g$ for $k$-vertex-colorings $f$ and $g$ of a graph $G$.
In Section \ref{sect:matching},  we extend Theorem \ref{thm:bip_upper} to the case of line perfect graphs and prove Theorem \ref{thm:bip_quad} by characterizing line perfect graphs such that the toric ideals of their matching polytopes are generated by quadratic binomials.
Finally, Section \ref{sect:non-perfect} considers examples of  $I_{\Mc_G}$ for general graphs
and Conjecture \ref{conj:non-perfect} for graphs with a small number of vertices, and discusses perfect matching polytopes.

\section{Preliminaries}
\label{sect:pre}
In this section, we introduce the toric ideal of a lattice polytope and define the stable set polytope associated with a graph. 
\subsection{Toric rings and toric ideals}
Let $\Pc \subset \RR^d_{\geq 0}$ be a lattice polytope with $\Pc \cap \ZZ^d_{\geq 0}=\{\ab_1,\ldots,\ab_n\}$ and
let $\KK[\tb,s]:=\KK[t_1,\ldots,t_d,s]$ be the polynomial ring in $d+1$ variables  over a field $\KK$.
Given a nonnegative integer vector 
$\ab=(a_1,\ldots,a_d) \in  \ZZ_{\geq 0}^d$, we write 
$\tb^{\ab}:=t_1^{a_1} t_2^{a_2}\cdots t_d^{a_d} \in \KK[\tb,s]$.
The \textit{toric ring} of $\Pc$ is 
\[
\KK[\Pc]:=\KK[\tb^{\ab_1} s,\ldots, \tb^{\ab_n} s] \subset \KK[\tb,s]. 
\]
We regard $\KK[\Pc]$ as a homogeneous algebra by setting each $\deg (\tb^{\ab_i} s)=1$.
Let $R[\Pc]=\KK[x_{1},\ldots,x_{n}]$ denote the polynomial ring in $n$ variables over  $\KK$ with each $\deg(x_i)=1$.
The \textit{toric ideal} $I_{\Pc}$ of $\Pc$ is the kernel of the surjective homomorphism 
$\pi:R[\Pc] \to \KK[\Pc]$ defined by $\pi(x_{i})=\tb^{\ab_i}s$ for $1 \leq i \leq n$.
Note that $I_{\Pc}$ is a prime ideal generated by homogeneous binomials.
The toric ring $\KK[\Pc]$ is called \textit{quadratic}
if $I_\Pc$ is generated by quadratic binomials.
We say that ``$I_\Pc$ is generated by quadratic binomials'' even if $I_\Pc =\{0\}$. In particular, $\omega(I_\Pc) \geq 2 $ and $\omega(\{0\})=2$.
The following is known.

\begin{Proposition}[{\cite{OHH}, \cite[Theorem 1.3]{OhsugiGeom}}]\label{faceprop}
    Let $F$ be a face of a lattice polytope $\Pc$.
    If ${\mathcal G}$ is a set of generators of $I_\Pc$,
    then ${\mathcal G} \cap R[F]$ is a set of generators of $I_F$.
    In particular, we have $\omega(I_F) \le \omega(I_\Pc)$.
\end{Proposition}

\subsection{Stable set polytopes and stable set ideals}
Let $G$ be a simple graph on $[d]$. A subset $S \subset [d]$ is called a {\em stable set} (or an {\em independent set}) of $G$
if $\{i,j\} \notin E(G)$ for all $i,j \in S$ with $i \neq j$.
In particular, the empty set $\emptyset$ and any singleton $\{i\}$ with $i \in [d]$
are stable.
Let $S(G)$ denote
the set of all stable sets of $G$.
Then the \textit{stable set polytope} $\Sc_G$ of $G$ is defined as
\[
\Sc_G=\conv\{\rho(S) : S \in S(G) \}.
\]
The matching polytope $\Mc_G$ is a stable set polytope of some graph.
Indeed, the \textit{line graph} $L(G)$ of $G$ is a simple graph whose vertex set is $E(G)$ and whose edge set is 
\[
\{ \{e, e'\} \subset \ E(G) : e \neq e' \mbox{ and } e \cap e' \neq \emptyset\}.
\]
Then one has $\Mc_G=\Sc_{L(G)}$ by changing coordinates.

For a graph $G$, the toric ideal $I_{\Sc_G}$ is called the \textit{stable set ideal} of $G$.
We can describe a system of generators of $I_{\Sc_G}$ in terms of $k$-vertex-colorings.
Given a graph $G$ on the vertex set $[d]$, and 
$\ab = (a_1,\ldots,a_d) \in \ZZ_{\ge 0}^d$,
let $G_\ab$ be the graph obtained from $G$ by replacing each vertex $i \in [d]$ 
with a complete graph $G^{(i)}$ of $a_i$ vertices (if $a_i =0$, then just delete the vertex $i$),
and joining all vertices $x \in G^{(i)}$ and $y \in G^{(j)}$ such that $\{i,j\}$ is an edge of $G$.
In particular, if $\ab =(1,\ldots,1)$, then $G_\ab = G$.
If $\ab = {\bf 0}$, then $G_\ab$ is the null graph (a graph without vertices). 
In addition, if $\ab$ is a $(0,1)$-vector, namely, $\ab \in \{0,1\}^d$, then $G_\ab$ is an induced subgraph of $G$.
We call $G_\ab$ a 
\textit{vertex-replication graph} of $G$.
Given a $k$-vertex-coloring $f$ of $G_{\ab}$ with $\ab \in \ZZ^d_{\geq 0}$,
we associate $f$ with a monomial
\[
\xb_f :=x_{S_{i_1}} \cdots x_{S_{i_k}} \in R[\Sc_G], 
\]
where
$S_{i_\ell} =  \{j \in [d] : G^{(j)} \cap f^{-1}(\ell) \ne \emptyset\}$
for $\ell=1,2,\ldots, k$.
Conversely, let $m=x_{S_{i_1}} \cdots x_{S_{i_k}} \in R[\Sc_G]$ be a monomial of degree $k$.
Then, for $\ab=(a_1,\ldots,a_n)$ with $a_p = |\{ \ell : p \in S_{i_\ell}\}|$,
there exists a $k$-vertex-coloring $f$ of $G_{\ab}$ such that $\xb_f = m$ (see \cite[Lemma 3.2]{OhsugiTsuchiya2023Kempe}).
Note that, for $k$-vertex-colorings $f$ and $g$ of an induced subgraph of $G$, 
$\xb_f = \xb_g$ if and only if $g$ is obtained from $f$ by permuting colors.
In this paper, we identify $f$ and $g$ if $g$ is obtained from $f$ by permuting colors.
%
%
Then we can describe a system of generators of $I_{\Sc_G}$ as follows:

\begin{Proposition}[{\cite[Theorem 3.3]{OhsugiTsuchiya2023Kempe}}]
\label{prop:gen}
Let $G$ be a simple graph on $[d]$. Then one has 
\[I_{\Sc_G} = \langle \xb_f- \xb_g : \mbox{$f$ and $g$ are $k$-vertex-colorings of $G_\ab$ with $\ab \in \ZZ^d_{\geq 0}$ and $k \geq \chi(G_{\ab})$}  \rangle,\]
where $\chi(G_{\ab})$ is the chromatic number of $G_{\ab}$.
\end{Proposition}

If $G'$ is an induced subgraph of a graph $G$, then $\Sc_{G'}$ is a 
face of $\Sc_G$.
From Proposition \ref{faceprop}, we have the following.

\begin{Proposition}
Let $G'$ be an induced subgraph of $G$.
Then we have $\omega (I_{\Sc_{G'}}) \le \omega (I_{\Sc_{G}})$.
\end{Proposition}

If $G'$ is a subgraph of $G$, then $L(G')$ is an induced subgraph of $L(G)$.
Hence the following fact holds from $\Mc_G = \Sc_{L(G)}$.

\begin{Proposition}\label{subgraph prop}
Let $G'$ be a subgraph of $G$.
Then we have $\omega (I_{\Mc_{G'}}) \le \omega (I_{\Mc_{G}})$.
\end{Proposition}

\section{A Bound on $\omega(I_{\Sc_G})$}
\label{sect:stable}
In this section, we prove Theorem \ref{thm:degree} by providing an upper bound on the maximal degree of minimal generators of the toric ideal of a stable set polytope using vertex-colorings.

Let $G$ be a graph. For a $k$-vertex-coloring of $G$ and a color $1 \leq j \leq k$, let $M(f,j)$ denote the set of all vertices of color $j$. 
We say that two $k$-vertex-colorings $f$ and $g$ of $G$ \textit{differ by an $m$-colored subgraph} if there is a set of colors $S$ of size $m$ such that $M(f,j) \neq M(g,j)$ for each $j \in S$, but $M(f,j)=M(g,j)$ for each $j \notin S$.
For two $k$-vertex-colorings $f,g$ of $G$, we write $f \sim_r g$ if there exists a sequence $f_0,f_1,\ldots,f_s$ of $k$-vertex-colorings of $G$ with $f_0=f$ and $f_s=g$ such that $f_i$ differs from $f_{i-1}$ by a $k_i$-colored subgraph with $k_i \leq r$. 
We give a proof of Theorem \ref{thm:degree} by showing the following more general result.
\begin{Theorem}\label{thm:degree_general}
Let $G$ be a graph on $[d]$.
Then $\omega(I_{\Sc_G}) \leq r$ if and only if for any $\ab \in \ZZ_{\geq 0}^d$ and for any $k$-vertex-colorings $f$ and $g$ of $G_{\ab}$, one has $f \sim_r g$.
\end{Theorem}

\begin{proof}
    (only if) 
Let $f$ and $g$ be $k$-vertex-colorings of $G_{\ab}$ with $\ab \in \ZZ_{\geq 0}^d$ and $k > r$.
    Then the binomial $F:=\xb_f-\xb_g$ belongs to $I_{\Sc_G}$. 
    If $F=0$, then $g$ is obtained from $f$ by permuting colors.
    Assume that $F \neq 0$.
    By the hypothesis, $F$ is generated by binomials of degree at most $r$ in $I_{\Sc_G}$.
    From the theory of binomial ideals \cite[Lemma 3.8]{HHObook}, there exists an expression
\begin{equation}
F = \sum_{i=1}^s {\bf x}^{\wb_i} (\xb_{f_i} - \xb_{g_i}),
\label{tenkai}
\end{equation}
where for each $i$, $f_i$ and $g_i$ are $k_i$-vertex-colorings of $G_{\ab_i}$ with $\ab_i \in \ZZ_{\geq 0}^d, k_i \leq r$ and 
$\xb_{f_i} - \xb_{g_i}$ $(\ne 0)$ is an irreducible binomial in $I_{\Sc_G}$.
We may suppose that $\xb_f=\xb^{\wb_1} \xb_{f_1}$ and $\xb^{\wb_s}\xb_{g_s}=\xb_g$.
Set $\xb_{f_1}=x_{S_1}x_{S_2}\cdots x_{S_{k_1}}, \xb_{g_1}=x_{S'_1}x_{S'_2}\cdots x_{S'_{k_1}}$ and $\xb^{\wb_1}=x_{T_1}x_{T_2}\cdots x_{T_{k-k_1}}$ with $S_i,S'_i,T_j \in S(G)$.
Since $\xb_{f_1}-\xb_{g_1} \in I_{\Sc_G}$, one has $S:=\bigcup_{1 \leq i \leq k_1} S_i=\bigcup_{1 \leq i \leq k_1} S'_i$ as multisets.
Moreover, it follows from $\xb_f-\xb^{\wb_1} \xb_{g_1} \in I_{\Sc_G}$ that there exists a $k$-vertex-coloring $g'_1$ of $G_{\ab}$ such that $\xb_{g'_1}=\xb^{\wb_1} \xb_{g_1}$.
Then by exchanging colors and exchanging the coloring of vertices in each clique $G^{(j)}$ of $G_{\ab}$ if necessary, we can assume that $M(f,j)\neq M(g'_1,j)$ for each $j \in S$ and $M(f,j)= M(g'_1,j)$ for each $j \notin S$.
This implies that $g'_1$ differs from $f$ by a $k_1$-colored subgraph.
By performing this process repeatedly, we can obtain a sequence $g_0',g_1',\ldots,g_s'$ of $k$-vertex-colorings of $G_{\ab}$ with $g_0'=f$ and 
$\xb_{g_s'}=\xb_g$
such that $g_i'$ differs from $g_{i-1}'$ by a $k_i$-colored subgraph with $k_i \leq r$.
Then $g$ is obtained from $g_s'$ by permuting colors.
Hence one has $f \sim_r g$.

(if) 
Let $F=\xb_f-\xb_g \in I_{\Sc_G}$
where $f$ and $g$ are 
$k$-vertex-colorings of $G_{\ab}$
with $\ab \in \ZZ_{\geq 0}^d$ and $k > r$.
From the assumption, there exists a sequence $f_0,f_1,\ldots,f_t$ of $k$-vertex-colorings of $G_{\ab}$ with $f_0=f$ and $f_t=g$ such that $f_i$ differs from $f_{i-1}$ by a $k_i$-colored subgraph with $k_i \leq r$.
Then there exists a set $S_i$ of colors
with $|S_i| = k_i$ such that $M(f_i,j)\neq M(f_{i-1},j)$ for each $j \in S_i$ and $M(f_{i},j)= M(f_{i-1},j)$ for each $j \notin S_i$.
Let $f_i|_{S_i}$ be the $k_i$-vertex-coloring of the induced subgraph of $G_{\ab}$ on the vertex set $f_i^{-1}(S_i)$
induced from $f_i$.
Then one has $\xb_{f_{i-1}|_{S_i}}-\xb_{f_{i}|_{S_i}} \in I_{\Sc_G}$.
Similarly, let $f_i|_{\overline{S_i}}$ be 
the $(k-k_i)$-vertex-coloring of the induced subgraph of $G_{\ab}$ on the vertex set $\{ j \in V(G_\ab) : f(j) \notin S_i\}$ induced from $f_i$.
Then we obtain 
$\xb_{f_{i-1}|_{\overline{S_i}}}=\xb_{f_i|_{\overline{S_i}}}, \xb_{f_{i-1}|_{\overline{S_i}}}\xb_{f_{i-1}|_{S_i}}=\xb_{f_{i-1}}$ and $\xb_{f_{i}|_{\overline{S_i}}}\xb_{f_{i}|_{S_i}}=\xb_{f_{i}}$.
Hence one has
\begin{align*}
F=\xb_f-\xb_g&=(\xb_{f_0}-\xb_{f_1})+(\xb_{f_1}-\xb_{f_2})+\cdots+(\xb_{f_{t-1}}-\xb_{f_t})\\
&=
\xb_{f_{1}|_{\overline{S_1}}}(\xb_{f_{0}|_{S_1}}-\xb_{f_{1}|_{S_1}})+\xb_{f_{2}|_{\overline{S_2}}}(\xb_{f_{1}|_{S_2}}-\xb_{f_{2}|_{S_2}})+\cdots+\xb_{f_{t}|_{\overline{S_t}}}(\xb_{f_{t-1}|_{S_t}}-\xb_{f_{t}|_{S_t}}).
\end{align*}
Since $\xb_{f_{i-1}|_{S_i}}-\xb_{f_{i}|_{S_i}} \in I_{\Sc_G}$ is a binomial of degree $k_i \leq r$, $F$ is generated by binomials of degree $\leq r$.
This implies $\omega(I_{\Sc_G}) \leq r$.
 \end{proof}

\begin{proof}[Proof of Theorem \ref{thm:degree}]
Let $G$ be a simple graph with $d$ edges and take $\ab \in \ZZ_{\geq 0}^d$.
Then the edge-replication multigraph $G^{(e)}_{\ab}$ coincides with the vertex-replication graph $L(G)_{\ab}$. Moreover, a $k$-edge-coloring $f$ of $G^{(e)}_{\ab}$ can be regard as a $k$-vertex-coloring of $L(G)_{\ab}$.
In particular, two $k$-edge-colorings $f$ and $g$ of $G^{(e)}_{\ab}$ differ by an $m$-colored subgraph if and only if $f$ and $g$ differ by an $m$-colored subgraph as $k$-vertex-colorings of $L(G)_{\ab}$. Hence Theorem \ref{thm:degree} follows from Theorem~\ref{thm:degree_general}.  
\end{proof}

\section{$r$-coloring ideals}\label{sect:method}
In this section, we give an algebraic method to determine if $f \sim_r g$ for two $k$-vertex-colorings $f,g$ of a graph $G$.

Given a $k$-vertex-coloring $f$ of $G$, and integers $1 \le i < j \le k$,
let $H$ be a connected component of the induced subgraph of $G$ 
on the vertex set $f^{-1}(i) \cup f^{-1}(j)$.
Then we can obtain a new $k$-vertex-coloring $g$ of $G$ by setting
$$
g(x) =
\begin{cases}
    f(x) & x \notin H,\\
    i & x \in H \mbox{ and } f(x) =j,\\
    j & x \in H \mbox{ and } f(x) =i.\\
\end{cases}
$$
We say that $g$ is obtained from $f$ by a {\em Kempe switching}.
Two $k$-vertex-colorings $f$ and $g$ of $G$ are called {\em Kempe equivalent}, if there exists
a sequence $f_0,f_1,\ldots,f_s$ of $k$-vertex-colorings of $G$ such that $f_0=f$, $f_s=g$,
and $f_i$ is obtained from $f_{i-1}$ by a Kempe switching.
It is easy to see that $f$ and $g$ are Kempe equivalent if and only if $f \sim_2 g$.
In \cite{OhsugiTsuchiya2024Kempe}, the third and fourth author introduced $2$-coloring ideals associated with graphs to examine when $f \sim_2 g$.
Given a graph $G$ on $[d]$, the \textit{$2$-coloring ideal} is defined as follows:
\begin{align*}
    J_{G,2}:=& \langle \xb_f - \xb_g : \mbox{$f$ and $g$ are $2$-colorings of $G_{\ab}$ with $\ab \in \{0,1\}^d$}\rangle\\
    =&\langle \xb_f - \xb_g : \mbox{$f$ and $g$ are $2$-colorings of an induced subgraph of $G$} \rangle \subset  R[\Sc_G].\\
\end{align*}
\begin{Proposition}[{\cite[Theorem 1.1]{OhsugiTsuchiya2024Kempe}}]
Let $G$ be a graph and take two $k$-vertex-colorings $f,g$ of $G$.
Then $f \sim_2 g$ if and only if $\xb_f -\xb_g \in J_{G,2}$.
\end{Proposition}

We generalize this result to determine $f \sim_r g$.
Given an integer $r \geq 2$, we define the \textit{$r$-coloring ideal} of $G$ as follows:
\begin{align*}
    J_{G,r}:=&\langle \xb_f - \xb_g : \mbox{$f$ and $g$ are $k$-vertex-colorings of an induced subgraph of $G$ with $k \leq r$} \rangle \subset R[\Sc_G].
\end{align*}
Note that $J_{G,r} \subset J_{G,r+1}$.
For $r \ge |V(G)|$, one has $J_{G,r}=J_{G,r+1}$.

\begin{Theorem}
Let $G$ be a graph and take two $k$-vertex-colorings $f,g$ of $G$.
Then $f \sim_r g$ if and only if $\xb_f -\xb_g \in J_{G,r}$.    
\end{Theorem}
\begin{proof}
   This follows by a similar argument as in the proof of Theorem \ref{thm:degree_general}.
\end{proof}

\section{A bound on $\omega(I_{\Mc_G})$ for a line perfect graph}
\label{sect:matching}
In this section, we give a proof of Theorem \ref{thm:bip_quad} by showing a more general result for line perfect graphs.

Let $G$ be a graph on $[d]$ with edge set $E(G)$. A subset $C \subset [d]$ is called \textit{clique} of $G$ if for any $i,j \in C$ with $i \neq j$, $\{i,j\} \in E(G)$. 
Let $\omega(G)$ denote the maximum cardinality of cliques of $G$.
A graph $G$ is called \textit{perfect} if every induced subgraph $H$ of $G$ satisfies $\chi(H)=\omega(H)$. Perfect graphs were introduced by Berge in \cite{Berge1960}.
A \textit{hole} is an induced cycle of length $\geq 5$ and an \textit{antihole} is the complement of a hole.
In \cite{strongperfect}, Chudnovsky, Robertson, Seymour and Thomas showed that a graph is perfect if and only if it has no odd holes and no odd antiholes. This result is called the strong perfect graph theorem.

A \textit{line perfect graph} is a graph whose line graph is perfect. 
Note that every bipartite graph is line perfect.
A characterization of line perfect graphs is known.
A vertex $v$ of a connected graph $G$ is called a \textit{cut vertex} if the graph obtained by the removal of $v$ from $G$ is disconnected. Given a graph $G$, a \textit{block} of $G$ is a maximal connected subgraph of $G$ with no cut vertices.
\begin{Proposition}[\cite{Maf, Tro}]
\label{lineperfect}
Let $G$ be a graph. Then the following conditions are equivalent{\rm :}
\begin{itemize}
    \item[\rm (i)] $G$ is line perfect{\rm ;} 
    \item[\rm (ii)] $G$ has no odd cycle of length $\ge 5$ as a subgraph{\rm ;}
    \item[\rm (iii)] each block of $G$ is either a bipartite graph, $K_4$, or $K_{1,1,n}$.
\end{itemize}
\end{Proposition}

The graph obtained by gluing two graphs at a clique $C$ of them is called a $|C|$-{\it clique sum} of them.
(Here we do not remove any edges of the clique.)
Clique sums of more than two graphs are defined by repeated application of this operation.

\begin{Proposition}[{\cite[Proposition 1]{MOS}}]
\label{stable_sum}
    Suppose that $G$ is a clique sum of graphs $G_1$ and $G_2$.
    Then one has $$\omega(I_{\Sc_G})= \max \{\omega(I_{\Sc_{G_1}}), \omega(I_{\Sc_{G_2}}) \}.$$
\end{Proposition}

\begin{Lemma}
\label{block_lemma}
    Let $G$ be a graph whose blocks are $H_1, \dots, H_s$.
    Then one has 
    $$\omega(I_{\Mc_G})= \max \{\omega(I_{\Mc_{H_1}}),\dots, \omega(I_{\Mc_{H_s}}) \}.$$
\end{Lemma}

\begin{proof}
Since $G$ is a $1$-clique sum of $H_1, \dots, H_s$, it is enough to show that
$$\omega(I_{\Mc_G})= \max \{\omega(I_{\Mc_{G_1}}), \omega(I_{\Mc_{G_2}}) \}$$
if $G$ is a $1$-clique sum of $G_1$ and $G_2$
at a vertex $v$.
Let $E_i = \{e \in E(G_i) : v \in e\}$ for $i=1,2$.
It then follows that $L(G)$ is a clique sum of $L(G_1), L(G_2)$ and the complete graph $L(G')$
along cliques $E_1$ and $E_2$,
where $G'$ is the graph whose edge set is $E_1 \cup E_2$.
Note that $I_{\Mc_{G'}} = I_{\Sc_{L(G')}} =\{0\}$.
Then
$$\omega(I_{\Mc_G})= \max \{\omega(I_{\Mc_{G_1}}), \omega(I_{\Mc_{G_2}}), \omega(I_{\Mc_{G'}})\}
= \max \{\omega(I_{\Mc_{G_1}}), \omega(I_{\Mc_{G_2}}) \}$$
from Proposition~\ref{stable_sum}.
\end{proof}

As a generalization of Theorem \ref{thm:degree} we give a bound of $\omega(I_{\Mc_G})$ for a line perfect graph $G$.
In fact,
\begin{Theorem}\label{thm:bound_line_perfect}
Let $G$ be a line perfect graph. Then one has $\omega(I_{\Mc_G}) \leq 3$.
\end{Theorem}

\begin{proof}
From Proposition \ref{lineperfect} and Lemma \ref{block_lemma},
it is enough to show that $\omega(I_{\Mc_G}) \leq 3$ if $G$ is a bipartite graph,
$K_4$, or $K_{1,1,n}$.

If $G$ is a bipartite graph, then $\omega(I_{\Mc_G}) \leq 3$ by Theorem \ref{thm:bip_upper}.
If $G=K_{1,1,n}$, then $G$ is obtained from $K_{2,n}$ by adding a new edge $e$.
Then
$\{e\}$ is a unique matching of $K_{1,1,n}$ that contains $e$.
Hence $I_{\Mc_G}$ and $I_{\Mc_{K_{2,n}}}$ have the same set of generators.
Thus we have $\omega(I_{\Mc_G}) = \omega(I_{\Mc_{K_{2,n}}}) \leq 3$.
Let $G=K_4$.
Then $L(G)$ has no induced path with 4 vertices.
Hence it is trivial that $L(G)$ is perfectly orderable, that is, there exists a linear order
$<$ on $V(L(G))$ such that no induced path with vertices $a,b,c,d$ and edges
$\{a,b\},\{ b,c\}, \{c,d\}$ satisfies $a< b$ and $d<c$.
It is known \cite[Theorem 3.1]{OhsugiShibataTsuchiya2023} that $\omega(I_{\Sc_{G'}}) = 2$ if $G'$
is perfectly orderable.
Thus $\omega(I_{\Mc_G}) = \omega(I_{\Sc_{L(G)}}) = 2$.
\end{proof}

Combining this theorem and Theorem \ref{thm:degree} we can obtain the following corollary.
\begin{Corollary}
    Let $G$ be a multigraph whose underlying simple graph is line perfect.
Then for any $k$-edge-colorings $f$ and $g$ of $G$, one has $f \sim_3 g$.
\end{Corollary}

Next, we characterize when $\omega(I_{\Mc_G})=2$ for a line perfect graph.
Bertschi introduced a hereditary class of perfect graphs in \cite{Bert}.
An \textit{even pair} in a graph $G$ is a pair of non-adjacent vertices of $G$ such that the length of all induced paths between them is even.
Contracting a pair of vertices $\{x,y\}$ in a graph $G$ means removing $x$ and $y$ and adding a new vertex $z$ with edges to every neighborhood of $x$ or $y$.
A graph $G$ is called \textit{even-contractile} if there exists a sequence $G_0,\ldots,G_k$ of graphs satisfying the following:
\begin{enumerate}[(i)]
    \item $G=G_0$;
    \item each $G_i$ is obtained from $G_{i-1}$ by contracting an even pair of $G_{i-1}$;
    \item $G_k$ is a complete graph.
\end{enumerate}
A graph $G$ is called \textit{perfectly contractile} if every induced subgraph of $G$ is even-contractile.
Every perfectly contractile graph is perfect.
In contrast to the strong perfect graph theorem, 
a forbidden graph characterization of perfectly contractile graphs is still open. However, there is a conjecture of this problem.
An \textit{odd prism} is a graph consisting of two disjoint triangles with three disjoint induced paths of odd length between them.
Everett and Reed conjectured that a graph $G$ is perfectly contractile if and only if $G$ contains no odd holes, no antiholes and no odd prisms as induced subgraphs. 
On the other hand, the third and fourth authors and Shibata gave the following conjecture. 
\begin{Conjecture}[{\cite[Conjecture 0.2]{OhsugiShibataTsuchiya2023}}]
\label{conj:second} Let $G$ be a perfect graph. 
Then the following conditions are equivalent{\rm :}
\begin{itemize}
    \item[\rm (i)] $G$ is perfectly contractile;
  \item[\rm (ii)] $G$ contains no odd holes, no antiholes and no odd prisms;
 \item[\rm (iii)] $\omega(I_{\Sc_G})=2$.
\end{itemize}
\end{Conjecture}
A graph $G$ is called \textit{line perfectly contractile} if its line graph $L(G)$ is perfectly contractile. 
If Conjecture \ref{conj:second} 
is true for a line perfect graph $G$, $\omega(I_{\Mc_G})=2$ if and only if $G$ is line perfectly contractile.
We show that this claim is true by proving the following theorem which 
implies Theorem \ref{thm:bip_quad}.
An {\it odd subdivision} of a graph $G$ is a graph obtained by replacing each edge of $G$ 
by a path of odd length. Note that $G$ itself is an odd subdivision of $G$.

\begin{Theorem} \label{line perfectcly contractile}
    Let $G$ be a line perfect graph. 
    Then the following conditions are equivalent{\em :}
    \begin{itemize}
        \item[{\rm (i)}]
        $\omega(I_{\Mc_G})=2${\rm ;} 
        \item[{\rm (ii)}]
        $G$ is line perfectly contractile{\rm ;} 
        
        \item[{\rm (iii)}]
        $L(G)$ has no odd prisms{\rm ;}
        \item[{\rm (iv)}]
        $G$ has no odd subdivision of $K_{2,3}$ as a subgraph{\rm ;} 
        \item[{\rm (v)}]
         each block of $G$ is either a bipartite graph having no odd subdivision of $K_{2,3}$ as a subgraph, $K_3$, $K_4$, or $K_{1,1,2}$.
    \end{itemize}
    Otherwise, $\omega(I_{\Mc_G})=3$.
\end{Theorem}

In order to prove this theorem, we recall the following results.
\begin{Proposition}[\cite{dartfree}]
\label{dartprop}
A dart-free graph is perfectly contractile if and only if it contains no odd holes, no antiholes and no odd prisms as induced subgraphs.
\end{Proposition}

\begin{Proposition}[{\cite[Theorem 1.5 (a)]{OhsugiTsuchiya2023Kempe}}]

\label{thm:app}
Let $G$ be a dart-free graph with no odd holes, no antiholes, and no odd prisms.
Then one has $\omega(I_{\Sc_G})=2$.
\end{Proposition}

Now, we give a proof of Theorem \ref{line perfectcly contractile}.

\begin{proof}[Proof of Theorem~\ref{line perfectcly contractile}]
From \cite[Theorem 1.7]{OhsugiShibataTsuchiya2023}, (i) $\Rightarrow$ (iii) holds for any perfect graph.

Let $G$ be a line perfect graph. 
Since $L(G)$ is perfect, it has no odd holes and odd antiholes.
In general, the line graph of a graph is dart-free (since claw-free) and 
has no graph $H$ below as an induced subgraph.
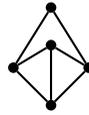
\begin{figure}[h]
\label{zu1}
    \centering
    \begin{tikzpicture}[dot/.style={circle,fill=black,minimum size=4pt,inner sep=0pt,outer sep=0pt},
                    line/.style={thick}]
    \node[dot] (A) at (0,0) {};
    \node[dot] (B) at (1,0) {};
    \node[dot] (C) at (0.5,0.8) {};
    \node[dot] (D) at (0.5,0.3) {};
    \node[dot] (E) at (0.5,-0.5) {};

    \draw[line] (A) -- (D) -- (B)--(E);
    \draw[line] (A) -- (E);
    \draw[line] (B) -- (C);
        \draw[line] (C) -- (A);
    \draw[line] (D) -- (E);
\end{tikzpicture}
\end{figure}

\noindent
Since the complement of $H$ is the disjoint union of an edge and a path with 3 vertices,
it follows that $L(G)$ has no antiholes of length $\ge 7$.
Note that an antihole of length $6$ is an odd prism.
Hence from Proposition~\ref{dartprop},
we have (ii) $\Leftrightarrow$ (iii).
Moreover, from Proposition~\ref{thm:app}, we have
    (iii) $\Rightarrow$ (i).
    It is known that the line graph of an odd subdivision of $K_{2,3}$ is an odd prism.
    (A subdivision of $K_{2,3}$ is called {\it theta}.)
    Thus we have (iii) $\Leftrightarrow$ (iv).
Finally, (iv) $\Leftrightarrow$ (v) follows from Proposition~\ref{lineperfect}.    
\end{proof}

\begin{Example}
    Let $G$ be an outerplanar bipartite graph.
    It is known that $G$ has no $K_{2,3}$ as a minor.
    Hence the toric ring of the matching polytope of $G$ is quadratic.
\end{Example}

\section{A bound on $\omega(I_{\Mc_G})$ for a general graph}
\label{sect:non-perfect}

In this section, we consider $I_{\Mc_G}$ for a general graph $G$.
First, we see examples of graphs $G$ with $\omega(I_{\Mc_G}) = 4$ by using Macaulay2 \cite{M2}.
\begin{Example}\label{exam:deg4}
    {\rm (1)
Let $G_1$ be the graph as follows:
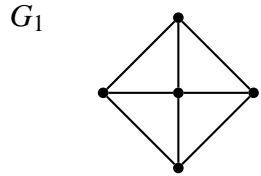
\begin{figure}[H]
\centering
		\begin{tikzpicture}[dot/.style={circle,fill=black,minimum size=4pt,inner sep=0pt,outer sep=0pt},
                    line/.style={thick}]
\node[dot] (1) at (-1,1) {};	
\node[dot] (2) at (0,2) {}; 
\node[dot] (4) at (1,1) {}; 
\node[dot] (3) at (0,0) {}; 
\node[dot] (5) at (0,1) {}; 

\draw[thick] (1)--(2);
\draw[thick] (2)--(4);
\draw[thick] (3)--(4);
\draw[thick] (3)--(1);
\draw[thick] (1)--(5);
\draw[thick] (2)--(5);
\draw[thick] (3)--(5);
\draw[thick] (4)--(5);
\node at (-2,2) {$G_1$};
\end{tikzpicture}
\end{figure}
Using Macaulay2, one has $\omega(I_{\Mc_{G_1}})=4$.
It then follows from Theorem \ref{thm:degree} that 
there exist $\ab \in \ZZ_{\geq 0}^8$ and $4$-edge-colorings $f$ and $g$ of $(G_1)^{(e)}_{\ab}$ such that $f \not\sim_3 g$.
In fact, for the following two $4$-edge-colorings $f$ and $g$ of $G_1$, one has $f \not\sim_3 g$.
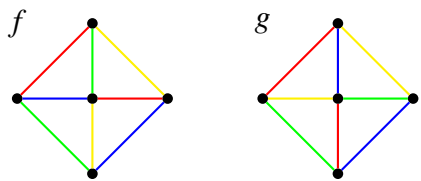
\begin{figure}[H]
\centering
	\begin{tikzpicture}[dot/.style={circle,fill=black,minimum size=4pt,inner sep=0pt,outer sep=0pt},
                    line/.style={thick}]
\node[dot] (1) at (-1,1) {};	
\node[dot] (2) at (0,2) {}; 
\node[dot] (4) at (1,1) {}; 
\node[dot] (3) at (0,0) {}; 
\node[dot] (5) at (0,1) {}; 

\draw[thick,red] (1)--(2);
\draw[thick,yellow] (2)--(4);
\draw[thick,blue] (3)--(4);
\draw[thick,green] (3)--(1);
\draw[thick,blue] (1)--(5);
\draw[thick,green] (2)--(5);
\draw[thick,yellow] (3)--(5);
\draw[thick,red] (4)--(5);
\node at (-1,2) {$f$};
\end{tikzpicture}
\ \ \ \ \ \ \ \ 
	\begin{tikzpicture}[dot/.style={circle,fill=black,minimum size=4pt,inner sep=0pt,outer sep=0pt},
                    line/.style={thick}]
\node[dot] (1) at (-1,1) {};	
\node[dot] (2) at (0,2) {}; 
\node[dot] (4) at (1,1) {}; 
\node[dot] (3) at (0,0) {}; 
\node[dot] (5) at (0,1) {}; 

\draw[thick,red] (1)--(2);
\draw[thick,yellow] (2)--(4);
\draw[thick,blue] (3)--(4);
\draw[thick,green] (3)--(1);
\draw[thick,yellow] (1)--(5);
\draw[thick,blue] (2)--(5);
\draw[thick,red] (3)--(5);
\draw[thick,green] (4)--(5);
\node at (-1,2) {$g$};
\end{tikzpicture}
\end{figure}

Indeed, for any three (resp. two) colors, the subgraph consisting of all edges with the colors has a unique $3$-edge-coloring (resp. $2$-edge-coloring) up to permuting colors.
This implies $f \not\sim_3 g$.

(2) Let $G_2$ be the graph as follows: 
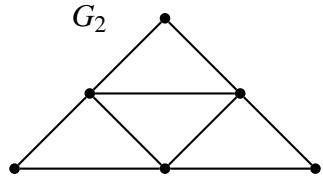
\begin{figure}[H]
\centering
	\begin{tikzpicture}[dot/.style={circle,fill=black,minimum size=4pt,inner sep=0pt,outer sep=0pt},
                    line/.style={thick}]
\node[dot] (1) at (0,2) {};	
\node[dot] (2) at (-1,1) {}; 
\node[dot] (3) at (1,1) {}; 
\node[dot] (4) at (-2,0) {}; 
\node[dot] (5) at (0,0) {}; 
\node[dot] (6) at (2,0) {};

\draw[thick] (1)--(2);
\draw[thick] (1)--(3);
\draw[thick] (2)--(3);
\draw[thick] (2)--(4);
\draw[thick] (2)--(5);
\draw[thick] (3)--(6);
\draw[thick] (3)--(5);
\draw[thick] (4)--(5);
\draw[thick] (5)--(6);
\node at (-1,2) {$G_2$};
\end{tikzpicture}
\end{figure}
Using Macaulay2, one has $\omega(I_{\Mc_{G_2}})=4$. Moreover, for the following two $4$-edge-colorings $f$ and $g$ of $G_2$, one has $f \not\sim_3 g$.
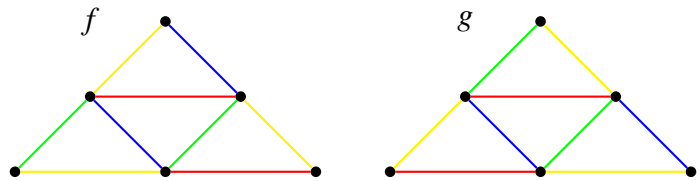
\begin{figure}[H]
\centering
	\begin{tikzpicture}[dot/.style={circle,fill=black,minimum size=4pt,inner sep=0pt,outer sep=0pt},
                    line/.style={thick}]
\node[dot] (1) at (0,2) {};	
\node[dot] (2) at (-1,1) {}; 
\node[dot] (3) at (1,1) {}; 
\node[dot] (4) at (-2,0) {}; 
\node[dot] (5) at (0,0) {}; 
\node[dot] (6) at (2,0) {};

\draw[thick,yellow] (1)--(2);
\draw[thick,blue] (1)--(3);
\draw[thick,red] (2)--(3);
\draw[thick,green] (2)--(4);
\draw[thick,blue] (2)--(5);
\draw[thick,yellow] (3)--(6);
\draw[thick,green] (3)--(5);
\draw[thick,yellow] (4)--(5);
\draw[thick,red] (5)--(6);
\node at (-1,2) {$f$};
\end{tikzpicture}\ \ \ \ \ \ \ \ 
	\begin{tikzpicture}[dot/.style={circle,fill=black,minimum size=4pt,inner sep=0pt,outer sep=0pt},
                    line/.style={thick}]
\node[dot] (1) at (0,2) {};	
\node[dot] (2) at (-1,1) {}; 
\node[dot] (3) at (1,1) {}; 
\node[dot] (4) at (-2,0) {}; 
\node[dot] (5) at (0,0) {}; 
\node[dot] (6) at (2,0) {};

\draw[thick,green] (1)--(2);
\draw[thick,yellow] (1)--(3);
\draw[thick,red] (2)--(3);
\draw[thick,yellow] (2)--(4);
\draw[thick,blue] (2)--(5);
\draw[thick,blue] (3)--(6);
\draw[thick,green] (3)--(5);
\draw[thick,red] (4)--(5);
\draw[thick,yellow] (5)--(6);
\node at (-1,2) {$g$};
\end{tikzpicture}
\end{figure}

(3)
Let $G_3$ be the graph as follows: 
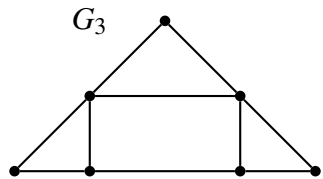
\begin{figure}[H]
\centering
	\begin{tikzpicture}[dot/.style={circle,fill=black,minimum size=4pt,inner sep=0pt,outer sep=0pt},
                    line/.style={thick}]
\node[dot] (1) at (0,2) {};	
\node[dot] (2) at (-1,1) {}; 
\node[dot] (3) at (1,1) {}; 
\node[dot] (4) at (-2,0) {}; 
\node[dot] (51) at (-1,0) {}; 
\node[dot] (52) at (1,0) {}; 
\node[dot] (6) at (2,0) {};

\draw[thick] (1)--(2);
\draw[thick] (1)--(3);
\draw[thick] (2)--(3);
\draw[thick] (2)--(4);
\draw[thick] (2)--(51);
\draw[thick] (3)--(6);
\draw[thick] (3)--(52);
\draw[thick] (4)--(51);
\draw[thick] (52)--(6);
\draw[thick] (51)--(52);
\node at (-1,2) {$G_3$};
\end{tikzpicture}
\end{figure}
Using Macaulay2, one has $\omega(I_{\Mc_{G_3}})=4$. Moreover, for the following two $4$-edge-colorings $f$ and $g$ of an edge-replication multigraph of $G_3$, one has $f \not\sim_3 g$.
Note that for any two $4$-edge-colorings $f_1,f_2$ of $G_3$, it follows that $f_1 \sim_3 f_2$. In particular, this example shows that computing $\omega(I_{\mathcal M_G})$ requires considering not only $G$ itself but also edge-replication multigraphs of $G$.
\begin{figure}[H]
\centering
	\begin{tikzpicture}[dot/.style={circle,fill=black,minimum size=4pt,inner sep=0pt,outer sep=0pt},
                    line/.style={thick}]
\node[dot] (1) at (0,2) {};	
\node[dot] (2) at (-1,1) {}; 
\node[dot] (3) at (1,1) {}; 
\node[dot] (4) at (-2,0) {}; 
\node[dot] (51) at (-1,0) {}; 
\node[dot] (52) at (1,0) {}; 
\node[dot] (6) at (2,0) {};

\draw[thick,yellow] (1)--(2);
\draw[thick,green] (1)--(3);
\draw[thick,red] (2)--(3);
\draw[thick,green] (2)--(4);
\draw[thick,blue] (2)--(51);
\draw[thick,yellow] (3)--(6);
\draw[thick,blue] (3)--(52);
\draw[thick,red] (4)--(51);
\draw[thick,red] (52)--(6);
\draw[thick,yellow, bend right] (51) to (52);
\draw[thick,green, bend left] (51) to (52);
\node at (-1,2) {$f$};
\end{tikzpicture}\ \ \ \ \ \ \ \ 
	\begin{tikzpicture}[dot/.style={circle,fill=black,minimum size=4pt,inner sep=0pt,outer sep=0pt},
                    line/.style={thick}]
\node[dot] (1) at (0,2) {};	
\node[dot] (2) at (-1,1) {}; 
\node[dot] (3) at (1,1) {}; 
\node[dot] (4) at (-2,0) {}; 
\node[dot] (51) at (-1,0) {}; 
\node[dot] (52) at (1,0) {}; 
\node[dot] (6) at (2,0) {};

\draw[thick,red] (1)--(2);
\draw[thick,green] (1)--(3);
\draw[thick,yellow] (2)--(3);
\draw[thick,blue] (2)--(4);
\draw[thick,green] (2)--(51);
\draw[thick,blue] (3)--(6);
\draw[thick,red] (3)--(52);
\draw[thick,red] (4)--(51);
\draw[thick,green] (52)--(6);
\draw[thick,yellow, bend right] (51) to (52);
\draw[thick,blue, bend left] (51) to (52);
\node at (-1,2) {$g$};
\end{tikzpicture}
\end{figure}
(4) Let $G_4,G_5,\ldots,G_8$ be the graphs as follows:
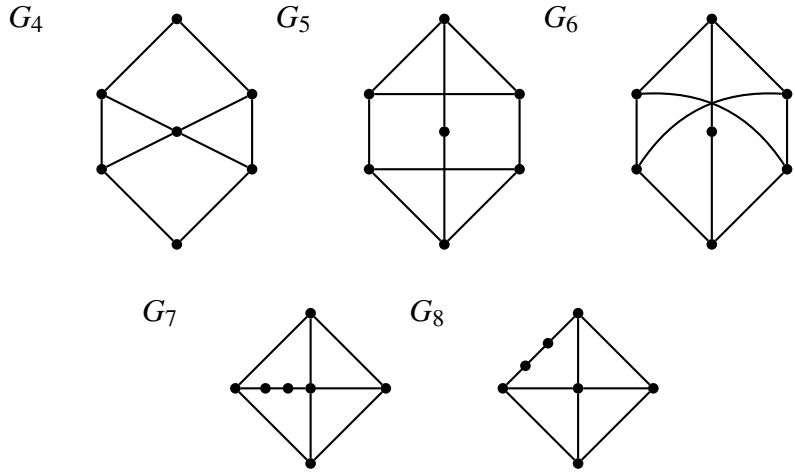
\begin{figure}[h]
\centering
\begin{tikzpicture}[dot/.style={circle,fill=black,minimum size=4pt,inner sep=0pt,outer sep=0pt},
                    line/.style={thick}]
\node[dot] (1) at (0,1) {};	
\node[dot] (2) at (-1,0) {}; 
\node[dot] (3) at (-1,-1) {}; 
\node[dot] (4) at (0,-2) {}; 
\node[dot] (5) at (1,-1) {};
\node[dot] (6) at (1,0) {};
\node[dot] (7) at (0,-0.5) {};

\draw[thick] (1)--(2)--(3)--(4)--(5)--(6)--(1);
\draw[thick] (2)--(7)--(5);
\draw[thick] (3)--(7)--(6);
\node at (-2,1) {$G_4$};
\end{tikzpicture}
	\begin{tikzpicture}[dot/.style={circle,fill=black,minimum size=4pt,inner sep=0pt,outer sep=0pt},
                    line/.style={thick}]
\node[dot] (1) at (0,1) {};	
\node[dot] (2) at (-1,0) {}; 
\node[dot] (3) at (-1,-1) {}; 
\node[dot] (4) at (0,-2) {}; 
\node[dot] (5) at (1,-1) {};
\node[dot] (6) at (1,0) {};
\node[dot] (7) at (0,-0.5) {};

\draw[thick] (1)--(2)--(3)--(4)--(5)--(6)--(1);
\draw[thick] (2)--(6);
\draw[thick] (3)--(5);
\draw[thick] (1)--(7)--(4);
\node at (-2,1) {$G_5$};
\end{tikzpicture}
	\begin{tikzpicture}[dot/.style={circle,fill=black,minimum size=4pt,inner sep=0pt,outer sep=0pt},
                    line/.style={thick}]
\node[dot] (1) at (0,1) {};	
\node[dot] (2) at (-1,0) {}; 
\node[dot] (3) at (-1,-1) {}; 
\node[dot] (4) at (0,-2) {}; 
\node[dot] (5) at (1,-1) {};
\node[dot] (6) at (1,0) {};
\node[dot] (7) at (0,-0.5) {};

\draw[thick] (1)--(2)--(3)--(4)--(5)--(6)--(1);
\draw[thick, bend left] (2) to (5);
\draw[thick, bend left] (3) to (6);
\draw[thick] (1)--(7)--(4);
\node at (-2,1) {$G_6$};
\end{tikzpicture}
\\ \ \\
\begin{tikzpicture}[dot/.style={circle,fill=black,minimum size=4pt,inner sep=0pt,outer sep=0pt},
                    line/.style={thick}]
\node[dot] (1) at (-1,1) {};	
\node[dot] (2) at (0,2) {}; 
\node[dot] (4) at (1,1) {}; 
\node[dot] (3) at (0,0) {}; 
\node[dot] (5) at (0,1) {}; 
\node[dot] (6) at (-0.3,1) {}; 
\node[dot] (7) at (-0.6,1) {};

\draw[thick] (1)--(2);
\draw[thick] (2)--(4);
\draw[thick] (3)--(4);
\draw[thick] (3)--(1);
\draw[thick] (1)--(5);
\draw[thick] (2)--(5);
\draw[thick] (3)--(5);
\draw[thick] (4)--(5);
\node at (-2,2) {$G_7$};
\end{tikzpicture}
\begin{tikzpicture}[dot/.style={circle,fill=black,minimum size=4pt,inner sep=0pt,outer sep=0pt},
                    line/.style={thick}]
\node[dot] (1) at (-1,1) {};	
\node[dot] (2) at (0,2) {}; 
\node[dot] (4) at (1,1) {}; 
\node[dot] (3) at (0,0) {}; 
\node[dot] (5) at (0,1) {}; 
\node[dot] (6) at (-0.7,1.3) {}; 
\node[dot] (7) at (-0.4,1.6) {};

\draw[thick] (1)--(2);
\draw[thick] (2)--(4);
\draw[thick] (3)--(4);
\draw[thick] (3)--(1);
\draw[thick] (1)--(5);
\draw[thick] (2)--(5);
\draw[thick] (3)--(5);
\draw[thick] (4)--(5);
\node at (-2,2) {$G_8$};
\end{tikzpicture}
\end{figure}

Using Macaulay2, one has $\omega(I_{\Mc_G})=4$ when $G \in \{G_4,\ldots,G_8\}$. We notice that the graph $G_7$ and $G_8$ are odd subdivisions of $G_1$.
}
\end{Example}
Proposition \ref{subgraph prop} says that,
for a graph $G$ and a subgraph $G'$ of $G$, one has $\omega(I_{\Mc_{G'}}) \leq \omega(I_{\Mc_G})$
since $\Mc_{G'}$ is a face of $\Mc_G$. 
In general, if $G'$ is an odd subdivision of a subgraph of $G$, we can obtain the same inequality.

\begin{Proposition}
Let $G$ be a graph, and let $G'$ be an odd subdivision of $G$.
Then $\Mc_{G'}$ has a face that is isomorphic to $\Mc_{G}$.
In particular, we have $\omega(I_{\Mc_{G}}) \leq \omega(I_{\Mc_{G'}})$.
\end{Proposition}

\begin{proof}
Let $E(G)=\{e_1,\dots,e_n\}$ be the edge set of $G$.
Suppose that $G'$ is obtained from $G$ by replacing an edge
$e_n$ of $G$ by a path $P=(e_{n}',e_{n+1}',e_{n+2}')$ with three edges.
Let $F_i$ be a face of $\Mc_{G'}$ defined by $F_i = \Mc_{G'} \cap H_i$
where
\begin{eqnarray*}
    H_1 &=&
\{
x \in \RR^{n+2}  :  x_n + x_{n+1}=1
\},\\
    H_2 &=&
\{
x \in \RR^{n+2}  : x_{n+1} + x_{n+2} =1
\}.
\end{eqnarray*}
Let $F = F_1 \cap F_2$.
Then $F$ is a face of $\Mc_{G'}$ and each vertex $x$ of $F$ satisfies
$(x_n,x_{n+1},x_{n+2}) \in \{ (0,1,0), (1,0,1)\}$.
It is easy to see that
$(x_1,\dots,x_{n-1},0,1,0) \in \Mc_{G'}$
if and only if
$(x_1,\dots,x_{n-1},0) \in \Mc_{G}$.
If $(x_1,\dots,x_{n-1},1,0,1) \in \Mc_{G'}$, then
$x_i = 0$ if $e_i$ is adjacent to either $e_n'$ or $e_{n+2}'$ in $G'$.
Note that $e_i$ is adjacent to either $e_n'$ or $e_{n+2}'$ in $G'$
if and only if $e_i$ is adjacent to $e_n$ in $G$.
Thus, $(x_1,\dots,x_{n-1},1,0,1) \in \Mc_{G'}$
if and only if
$(x_1,\dots,x_{n-1},1) \in \Mc_{G}$.
Let $\varphi : \RR^{n+2} \rightarrow \RR^{n+2} $ be a linear transformation 
defined by $\varphi(x) = (x_1,\dots,x_n, x_{n+1} +x_n ,x_{n+2}-x_n)$.
Then $\varphi(F) = \Mc_G \times \{ (1,0) \} \cong \Mc_G $.

Since every odd subdivision is obtained from $G$
by replacing an edge
of $G$ by a path with three edges repeatedly,
we have a desired conclusion from this.
\end{proof}

From Propositions \ref{faceprop} and \ref{subgraph prop}, we have the following.

\begin{Corollary}
Suppose that a graph $G_0$ contains an odd subdivision $G'$ of a graph $G$ as a subgraph.
Then we have
$\omega(I_{\Mc_{G}}) \leq \omega(I_{\Mc_{G'}}) \le \omega(I_{\Mc_{G_0}})$.
\end{Corollary}

Hence we obtain the following from Proposition \ref{subgraph prop}.

\begin{Proposition}
    Let $G$ be a graph and let $G_1, \ldots, G_5$ and $G_6$ be the graphs as in Example \ref{exam:deg4}. 
    If $G$ contains an odd subdivision of $G_1, \ldots, G_5$ or $G_6$ as a subgraph, then one has $\omega(I_{\Mc_G}) \geq 4$.
\end{Proposition}

Next, we consider Conjecture \ref{conj:non-perfect}.
Note that to solve the conjecture, it suffices to consider case of complete graphs from Proposition \ref{subgraph prop}. In other words, Conjecture \ref{conj:non-perfect} is equivalent to the following.
\begin{Conjecture}
    Let $K_d$ be a complete graph with $d$ vertices. Then one has $\omega(I_{\Mc_{K_d}}) \leq 4$.
\end{Conjecture}

For graphs with a small number of vertices, we consider Conjecture \ref{conj:non-perfect}.
Since $K_4$ is line perfect and has no $K_{2,3}$ as a subgraph, we have the following from Theorem \ref{line perfectcly contractile}.

\begin{Corollary}\label{d4}
     Let $G$ be a graph on $[d]$.
     If $d \leq 4$, then $\omega(I_{\Mc_G}) = 2$.
\end{Corollary}

For graphs with $5$, $6$ and $7$ vertices, we verify Conjecture \ref{conj:non-perfect} through computational experiments. 
By using Macaulay2,
we can confirm that $\omega(I_{\Mc_{K_7}}) = 4$.
Hence, for graphs $G$ with $5$, $6$ and $7$ vertices,
one has $\omega(I_{\Mc_G}) \leq 4$. Additionally, we classify all graphs $G$ such that $\omega(I_{\Mc_G})=4$ by using both Macaulay2 and Nauty \cite{nauty}. 
    Let $G$ be a graph on $[d]$ and let $G_1,G_2,\ldots,G_8$ be the graphs as in Example \ref{exam:deg4}, respectively. Then one has
    \begin{enumerate}
        \item[{\rm (1)}] Assume $d = 5$. Then $\omega(I_{\Mc_G})=4$ if and only if $G$ contains $G_1$ as a subgraph. Otherwise, $\omega(I_{\Mc_G}) \leq 3$.
        \item[{\rm (2)}] Assume $d = 6$. Then $\omega(I_{\Mc_G})=4$ if and only if $G$ contains $G_1$ or $G_2$ as a subgraph. Otherwise, $\omega(I_{\Mc_G}) \leq 3$.
         \item[{\rm (3)}]  Assume $d = 7$. Then $\omega(I_{\Mc_G})=4$ if and only if $G$ contains one of $G_1,G_2,\ldots,G_{8}$ as a subgraph. Otherwise, $\omega(I_{\Mc_G}) \leq 3$.
    \end{enumerate}


Next, we consider Conjecture \ref{conj:non-perfect} for a class of graphs.
For $d \geq 4$, let $W_d$ be the graph on $[d]$ whose edge set is 
\[
\{ \{1,2\},\{2,3\},\ldots,\{d-2,d-1\},\{1,d-1\}\} \cup \{\{1,d\},\{2,d\},\ldots,\{d-1,d\}\}.
\]

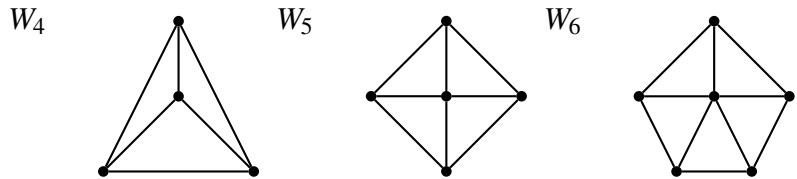
\begin{figure}[h]
\centering
\begin{tikzpicture}[dot/.style={circle,fill=black,minimum size=4pt,inner sep=0pt,outer sep=0pt},
                    line/.style={thick}]
\node[dot] (1) at (0,1) {};	
\node[dot] (2) at (-1,-1) {}; 
\node[dot] (3) at (1,-1) {}; 
\node[dot] (4) at (0,0) {}; 

\draw[thick] (1)--(2);
\draw[thick] (1)--(3);
\draw[thick] (2)--(3);
\draw[thick] (1)--(4);
\draw[thick] (2)--(4);
\draw[thick] (3)--(4);
\node at (-2,1) {$W_4$};
\end{tikzpicture}
		\begin{tikzpicture}[dot/.style={circle,fill=black,minimum size=4pt,inner sep=0pt,outer sep=0pt},
                    line/.style={thick}]
\node[dot] (1) at (-1,1) {};	
\node[dot] (2) at (0,2) {}; 
\node[dot] (4) at (1,1) {}; 
\node[dot] (3) at (0,0) {}; 
\node[dot] (5) at (0,1) {}; 

\draw[thick] (1)--(2);
\draw[thick] (2)--(4);
\draw[thick] (3)--(4);
\draw[thick] (3)--(1);
\draw[thick] (1)--(5);
\draw[thick] (2)--(5);
\draw[thick] (3)--(5);
\draw[thick] (4)--(5);
\node at (-2,2) {$W_5$};
\end{tikzpicture}
\begin{tikzpicture}[dot/.style={circle,fill=black,minimum size=4pt,inner sep=0pt,outer sep=0pt},
                    line/.style={thick}]
\node[dot] (1) at (0,1) {};	
\node[dot] (2) at (-1,0) {}; 
\node[dot] (3) at (-0.5,-1) {}; 
\node[dot] (5) at (1,0) {}; 
\node[dot] (4) at (0.5,-1) {}; 
\node[dot] (6) at (0,0) {}; 
\draw[thick] (1)--(2);
\draw[thick] (2)--(3);
\draw[thick] (3)--(4);
\draw[thick] (4)--(5);
\draw[thick] (5)--(1);
\draw[thick] (1)--(6);
\draw[thick] (2)--(6);
\draw[thick] (3)--(6);
\draw[thick] (4)--(6);
\draw[thick] (5)--(6);

\node at (-2,1) {$W_6$};
\end{tikzpicture}
\end{figure}

The graph $W_d$ is called a \textit{wheel graph}. Note that $W_5$ is the graph $G_1$ as in Example \ref{exam:deg4}. 
If $d$ is odd, then $W_d$ contains an odd subdivision of $W_5$ as a subgraph.
In this case, one has $\omega(I_{\Mc_{W_d}}) \geq 4$.
From Corollary \ref{d4}, $\omega(I_{\Mc_{W_4}})=2$. Moreover, by using Macaulay2, we obtain $\omega(I_{\Mc_{W_6}})=2$.
This leads the following conjecture.
\begin{Conjecture}
    Let $d$ be an integer $\geq 4$. Then one has
    \[
\omega(I_{\Mc_{W_d}})=\begin{cases}
    2 & \mbox{if $d$ is even},\\
    4 & \mbox{if $d$ is odd}.
\end{cases}
    \]
\end{Conjecture}
By computational experiments, we can confirm this conjecture for $d \leq 10$.

Finally, we discuss perfect matching polytopes.
Since $\Pc_G$ is a face of $\Mc_G$ for any graph $G$, 
we have the following from Theorem \ref{thm:degree}
and Proposition \ref{faceprop}.

\begin{Proposition}\label{pm gene}
    Let $G$ be a graph with $n$ edges.
Then $\omega(I_{\Pc_G}) \leq r$ if and only if for any $\ab \in \ZZ_{\geq 0}^n$ and for any $k$-edge-colorings $f$ and $g$ of $G^{(e)}_{\ab}$ such that each color of $f$ and $g$ corresponds to a perfect matching of $G$, one has $f \sim_r g$.
\end{Proposition}

\begin{Proposition}\label{p to m}
    Let $G$ be a graph on $[d]$ with $n$ edges.
    Then we have the following:
\begin{itemize}
    \item[{\rm (a)}]
    $\omega (I_{\Pc_G})  \le \omega (I_{\Mc_G}) ${\rm ;}
    \item[{\rm (b)}]
    There exists a graph $G'$ on $[2d]$ with $2n+d$ edges such that $\omega (I_{\Mc_G})  \le \omega (I_{\Pc_{G'}}) $.
\end{itemize}
    
\end{Proposition}

\begin{proof}
(a)
Since $\Pc_G$ is a face of $\Mc_G$, 
the assertion follows from Proposition \ref{faceprop}.

(b)
The idea of the proof comes from the argument appearing in e.g., \cite[p.129]{BONAMY2023126}.
Let $G'$ be a graph on $[2d]$ obtained by 
taking two copies of $G$ and adding edges between each vertex and its copy.
Let $r = \omega (I_{\Pc_{G'}})   $ and assume that $r < \omega (I_{\Mc_G}) $.
From Theorem \ref{thm:degree}, there exists 
$\ab \in \ZZ_{\geq 0}^n$ and $k$-edge-colorings $f$ and $g$ of $G^{(e)}_{\ab}$, 
such that $f \not\sim_r g$.
Let $\ab' = (\ab, \ab, \bb) \in \ZZ_{\ge 0}^{2n+d}$ where $\bb =(b_1,\dots,b_d)$
with $b_i = k - \deg_{G_\ab} (i) $.
Then ${G'}^{(e)}_{\ab'}$ is a $k$-regular multigraph.
We define a $k$-edge coloring $f'$ of ${G'}^{(e)}_{\ab'}$ as follows:
\begin{itemize}
    \item 
If $e$ is an edge of one of copies of $G$, then we set $f'(e) = f(e)$.
\item 
If $e_1,\dots, e_s$ with $s= k -  \deg_{G_\ab} (i) $
are edges between the vertex $i$ and its copy, then colors
$f'(e_1),\dots, f'(e_s)$ are distinct each other
and different from $\{f(e) : e \mbox{ is adjacent to }i \mbox{ in } G\}$.
\end{itemize}
Similarly, we define $g'$ from $g$.
Then each color of $f'$ (resp. $g'$) is a perfect matching of $G'$.
From Proposition \ref{pm gene}, we have $f' \sim_r g'$.
Since $G^{(e)}_\ab$ is an induced subgraph of ${G'}^{(e)}_{\ab'}$,  
it follows that $f \sim_r g$, a contradiction.
\end{proof}

Thus we can obtain the following.
\begin{Proposition}\label{prop:conjs}
       Conjecture \ref{conj:non-perfect} is equivalent to Conjecture \ref{conj:non-perfect-matching}.
\end{Proposition}

\appendix
\section{Flow polytopes}
\label{sect:app}

Let $Q$ be a directed graph with the vertex set $Q_0$
and the arrow set $Q_1$.
For an arrow $a \in Q_1$, let
$a^-$ be the starting vertex of $a$,
and let 
$a^+$ be the terminating vertex of $a$.
Given an integer vector $\theta \in \ZZ^{Q_0}$ and 
non-negative integer vectors ${\bf l}, {\bf u} \in \ZZ_{\ge 0}^{Q_1}$,
the \textit{flow polytope} associated with $Q, \theta, {\bf l}, {\bf u}$ is the polytope
$$
\nabla(Q, \theta, {\bf l}, {\bf u})
=
\left\{
x \in \RR^{Q_1} : {\bf l} \le x \le {\bf u},
\theta (v) =\sum_{a^+ =v} x(a) - \sum_{a^- =v} x(a) 
\mbox{ for } \forall v \in Q_0 
\right\}.
$$

We see that the matching polytope of a bipartite graph is a flow polytope.
\begin{Proposition}
For any bipartite graph $G$, the matching polytope $\Mc_G$ is isomorphic to a flow polytope.
\end{Proposition}
\begin{proof}
Let $G$ be a bipartite graph on the vertex set $V = V_1 \sqcup V_2$
and the edge set $E$.
It is known that the matching polytope $\mathcal{M}_G$ of $G$
coincides with 
$$
\left\{
x \in \RR^E : 
\begin{array}{cl}
x(e) \ge 0  & \mbox{for } \forall e \in E\\
\sum_{e \ni v} x(e) \le 1 & \mbox{for } \forall v \in V
\end{array}
\right\}.
$$
Hence we have
$$
{\mathcal M}_G
=
\left\{
x \in \RR^E : 
\begin{array}{cl}
x(e) \ge 0  & \mbox{for } \forall e \in E\\
y(v) + \sum_{e \ni v} x(e) = 1 & \mbox{for } \forall v \in V\\
y(v) \ge 0 & \mbox{for } \forall v \in V
\end{array}
\right\}.
$$
Thus $\mathcal{M}_G$ is a projection of the polytope
$$
\mathcal{P} = 
\left\{
(x, y) \in \RR^{E \cup V} : 
\begin{array}{cl}
x(e) \ge 0  & \mbox{for } \forall e \in E\\
y(v) + \sum_{e \ni v} x(e) = 1 & \mbox{for } \forall v \in V\\
y(v) \ge 0 & \mbox{for } \forall v \in V
\end{array}
\right\}.
$$
Note that, if $(x,y) \in {\mathcal P}$, then we have
$$
\sum_{v \in V_2} y(v) - 
\sum_{v \in V_1} y(v) =
\sum_{v \in V_2} \left(1 - \sum_{e \ni v} x(e) \right)
-
\sum_{v \in V_1} \left(1 - \sum_{e \ni v} x(e)\right)
= |V_2| - |V_1|.
$$
Hence
${\mathcal P}$ is the flow polytope $\nabla(Q, \theta, {\bf 0}, {\bf 1})$
where $Q$ is the directed graph
on the vertex set $Q_0= V_1 \cup V_2 \cup \{v_0\}$ and the arrow set
$$Q_1=
\{
(i,j) : i \in V_1, j \in V_2
\}
\cup 
\{
(i,v_0) : i \in V_1
\}
\cup 
\{
(v_0,j) : j \in V_2
\},
$$
and $\theta = (-1,\dots,-1, 1,\dots, 1, |V_2| - |V_1|) \in \ZZ^{V_1 \cup V_2 \cup \{v_0\}}$.
Since ${\mathcal M}_G$ is a projection of ${\mathcal P}$ with
$$
\dim {\mathcal P} = 
|Q_1| - |Q_0| +1
=
(|E|+|V|) - (|V|+1) +1 = |E|
=
\dim {\mathcal M}_G,$$
it follows that ${\mathcal M}_G$ is isomorphic to ${\mathcal P}=\nabla(Q, \theta, {\bf 0}, {\bf 1})$.
\end{proof}

\subsection*{Acknowledgment}
This work was supported by 
JSPS KAKENHI 22K13890, 24K00534, 26K00618 and 26K16970.

\bibliographystyle{plain}
\bibliography{bibliography}
\end{document}